\documentclass[final,1p,times]{elsarticle}

\usepackage[colorlinks=true]{hyperref}
\usepackage{amsfonts}
\usepackage{mathrsfs}
\usepackage{amsthm,amsmath,amssymb,mathrsfs,amsfonts,graphicx,graphics,latexsym,exscale,cmmib57,dsfont,bbm,amscd,ulem,curves}
\usepackage{color,soul}
\usepackage{makeidx}
\makeindex
\newtheorem{thm}{Theorem}[section]
\newtheorem{lem}[thm]{Lemma}

\newtheorem{cor}[thm]{Corollary}

\newtheorem{prop}[thm]{Proposition}

\theoremstyle{definition}
\newtheorem*{note*}{Note}
\newtheorem{note}{Note}
\newtheorem{exa}[thm]{Example}

\newtheorem{remark}[thm]{Remark}

\newtheorem{sn}[thm]{Definition}

\newtheorem{snAus}[thm]{Standing assumption}

\newcommand{\bt}{\boldsymbol{t}}

\journal{xxx}

\begin{document}

\begin{frontmatter}
\title{On transitivity dynamics of topological semiflows}

\author{Joseph Auslander}
\ead{jna@math.umd.edu}
\address{Department of Mathematics, University of Maryland, College Park, MD 20742, U.S.A.}

\author{Xiongping Dai}
\ead{xpdai@nju.edu.cn}
\address{Department of Mathematics, Nanjing University, Nanjing 210093, People's Republic of China}

\begin{abstract}
Let $T\times X\rightarrow X, (t,x)\mapsto tx$, be a topological semiflow on a topological space $X$ with phase semigroup $T$. We introduce and discuss in this paper various transitivity dynamics of $(T,X)$.
\end{abstract}

\end{frontmatter}

\section{Introduction}\label{sec0}
Let $T$ be a topological semigroup with a neutral element $e$ and $X$ a non-singleton topological space. We consider a \textit{semiflow} $T\times X\rightarrow X,\ (t,x)\mapsto tx$, denoted by $(T,X)$, which satisfies conditions:
\begin{enumerate}
\item $(t,x)\mapsto tx$ is jointly continuous;
\item $ex=x\ \forall x\in X$; and
\item $t(sx)=(ts)x\ \forall s,t\in T, x\in X$.
\end{enumerate}
Here $T$ is called the \textit{phase semigroup} and $X$ the \textit{phase space}. Moreover,
\begin{enumerate}
\item[a).] when $T$ is a group, we will call $(T,X)$ a \textit{flow} with phase group $T$;
\item[b).] if each $t\in T$ is surjective, then $(T,X)$ will be said to be \textit{surjective};
\item[c).] if each $t\in T$ is bijective, then $(T,X)$ will be called \textit{invertible} and $\langle T\rangle$ stands for the smallest group of self-homeomorphisms of $X$ containing $T$.
\end{enumerate}

Clearly a flow is invertible and an invertible semiflow is certainly surjective. Let $TA=\bigcup_{t\in T}tA$ for all $A\subseteq X$.
If $(T,X)$ is surjective, then $TX=X$; however, the converse need not be true. For example, let $X=\{a,b\}$ a discrete space, $f\colon a\mapsto b\mapsto b$ and $g\colon b\mapsto a\mapsto a$; and let $T=\langle f,g\rangle_+$ be the semigroup generated by $f$ and $g$. Then $TX=X$ but $f(X)=\{b\}\not=X$.

\begin{sn}\label{def1.1}
$(T,X)$ is called \textit{topologically transitive} (TT) if for all non-empty open sets $U,V$ in $X$, the set
\begin{equation*}
N_T(U,V)=\{t\in T\,|\,U\cap t^{-1}V\not=\emptyset\}
\end{equation*}
is non-empty. Clearly, $N_T(U,V)=\{t\in T\,|\,tU\cap V\not=\emptyset\}$. See Theorem~\ref{thm2.1} for two equivalent descriptions of TT.
\end{sn}

\begin{sn}\label{def1.2}
$(T,X)$ is called \textit{point-transitive} (PT) if the set
\begin{equation*}
\mathrm{Tran}\,(T,X)=\{x\in X\,|\,Tx\textrm{ is dense in }X\}
\end{equation*}
is non-empty. Every point of $\mathrm{Tran}\,(T,X)$ is called a \textit{transitive point} of $(T,X)$.
\end{sn}

Both of TT and PT are the most important and basic transitivity of semiflows.
If the phase space $X$ is not separable, then $\mathrm{TT}\not\Rightarrow \mathrm{PT}$ in general even for $T=\mathbb{Z}$ (cf.~\cite[Example~4.17]{EE}). However, if $X$ is a Polish space, then $(T,X)$ is TT iff $\mathrm{Tran}\,(T,X)$ is a dense $G_\delta$ set of $X$ (cf.~\cite{KM} and \cite[Basic fact~1]{DT}).

Let $(T,X)$ be a flow with phase group $T$. Then by $\textrm{cls}_XTsx=\textrm{cls}_XTx$ for all $s\in T$ and $x\in X$, $\mathrm{Tran}\,(T,X)$ is invariant so if $(T,X)$ is PT, $\mathrm{Tran}\,(T,X)$ is dense in $X$ and $(T,X)$ is thus TT.

However, $\mathrm{PT}\not\Rightarrow\mathrm{TT}$ in general if $(T,X)$ is not a flow even for $X$ is a compact metric space. Let us consider a simple counterexample as follows.

\begin{exa}\label{exa1.3}
Let $T=(\mathbb{Q}_+,+)$ the nonnegative rational numbers semigroup and $X=\mathbb{R}_+\cup\{+\infty\}$ with the usual one-point compactification topology. Define
\begin{equation*}
T\times X\rightarrow X,\quad (t,x)\mapsto t+x.
\end{equation*}
Clearly, $\mathrm{Tran}\,(T,X)=\{0\}$ so $(T,X)$ is PT. But $(T,X)$ is not TT, for $\mathrm{Tran}\,(T,X)$ is not a dense $G_\delta$-subset of $X$. In fact, for open subsets $U=(10,12)$ and $V=(5,6)$ of $X$, we have $N_T(U,V)=\emptyset$.
\end{exa}

As mentioned before, if $(T,X)$ is PT with $\mathrm{Tran}\,(T,X)$ dense in $X$ then $(T,X)$ is TT. But this sufficient condition is not a necessary one for TT with non-Polish phase space. In this note, we will consider the following question:
\begin{quote}
{\it When does $\mathrm{Tran}\,(T,X)\not=\emptyset$ imply TT for a semiflow $(T,X)$ with $T$ not a group?}
\end{quote}

We will also consider other kinds of important transitivity dynamics and their relationships with TT and PT in this paper.

In particular, in $\S\ref{sec4}$ we shall consider ``universally transitive'' semiflow with compact Hausdorff phase space and show that a universally transitive PT semiflow is minimal distal (cf.~Definition~\ref{def4.2}.2 and Theorem~\ref{thm4.16}).

In $\S\ref{sec6}$ we shall consider ``syndetically transitive'' semiflow and show that a non-minimal syndetically transitive semiflow is ``syndetically sensitive'' (cf.~Definition~\ref{def6.1} and Theorem~\ref{thm6.12}). This will generalizes some results of Glasner and Weiss on \textit{E}-systems.
\section{Prolongation and nonwandering points}\label{sec2}

Let $(T,X)$ be a semiflow on a topological space $X$ and $A\subset X$. We say $A$ is \textit{invariant} if $Tx\subseteq A$ for all $x\in A$, i.e., $TA\subseteq A$; and $A$ is called \textit{negatively}- or \textit{$T^{-1}$-invariant} if $t^{-1}x\subseteq A$ for all $x\in A$ and $t\in T$, i.e., $T^{-1}A\subseteq A$. By $\textrm{Int}_XA$ we will denote the interior of $A$ relative to $X$.

In this section we will consider basic properties of TT. We can easily verify the following two elementary sufficient and necessary conditions for TT of semiflow.

\begin{thm}\label{thm2.1}
Let $(T,X)$ be a semiflow on a topological space $X$ with phase semigroup $T$. Then the following conditions are pairwise equivalent.
\begin{enumerate}
\item[$(1)$] $(T,X)$ is TT.
\item[$(2)$] Every invariant set with non-empty interior is dense in $X$.
\item[$(3)$] Every non-empty, open, and $T^{-1}$-invariant set is dense in $X$.
\end{enumerate}
\end{thm}

\begin{proof}
Condition $(1)\Rightarrow(2)$. To be contrary assume there is an invariant set $X_0$ with $\textrm{Int}_XX_0\not=\emptyset$ such that $\textrm{cls}_XX_0\not=X$.
Let $U=\textrm{Int}_XX_0$ and set $V=X\setminus\textrm{cls}_XX_0$; then $U\not=\emptyset$ and $V\not=\emptyset$ such that $N_T(U,V)=\emptyset$, a contradiction to TT.

Condition $(2)\Rightarrow(1)$. Let $U,V$ be two non-empty open subsets of $X$. Since $e\in T$ so that $U\subseteq TU$, $TU$ is an invariant set with non-empty interior in $X$. Hence $TU\cap V\not=\emptyset$ so that $tU\cap V\not=\emptyset$ for some $t\in T$. This implies that $U\cap t^{-1}V\not=\emptyset$ for some $t\in T$. Thus $(T,X)$ is TT.

Condition $(3)\Rightarrow(1)$. Given non-empty open sets $U,V$ in $X$, since $T^{-1}V$ is $T^{-1}$-invariant open non-empty, $T^{-1}V$ is dense in $X$ so $U\cap t^{-1}V\not=\emptyset$ for some $t\in T$. Thus $(T,X)$ is TT.

Condition $(1)\Rightarrow(3)$. Let $U$ be a non-empty open $T^{-1}$-invariant subset of $X$. If $V=X-\textrm{cls}_XU$ were not empty, then $N_T(V,U)=\emptyset$ and this contradicts TT. Thus $U$ must be dense in $X$.

This proves Theorem~\ref{thm2.1}.
\end{proof}

So, if $X$ contains a non-dense open $T$- or $T^{-1}$-invariant non-empty set, then $(T,X)$ is not TT. Moreover, if $(T,X)$ is TT with $\textrm{Int}_XTx\not=\emptyset$, then $x\in\mathrm{Tran}\,(T,X)$.

In (2) $\Rightarrow$ (1) of Theorem~\ref{thm2.1}, the neutral element $e$ plays a role.
By Theorem~\ref{thm2.1} we can obtain the following well-known sufficient and necessary condition.

\begin{cor}\label{cor2.2}
Let $(T,X)$ be a flow; then $(T,X)$ is TT if and only if every invariant open non-empty subset is dense in $X$.
\end{cor}

The following notion is a kind of generalized orbit closure of a semiflow $(T,X)$ on a topological space $X$.

\begin{sn}\label{def2.3}
Let $(T,X)$ be a semiflow with $T$ a locally compact non-compact topological semigroup. Let $\mathscr{K}_e$ be the collection of compact neighborhoods of $e$ in $T$ and $x\in X$.
\begin{enumerate}
\item[(1)] We say $y\in X$ is in the \textit{prolongation} of $x$, denoted by $y\in \Omega_T(x)$,\index{$\Omega_T(x)$} if
    for all neighborhoods $U$ of $x$ and $V$ of $y$ in $X$ and all $K\in\mathscr{K}_e$ in $T$, one can find some $x^\prime\in U$ and $t\in T\setminus K$ such that $tx^\prime\in V$.
\item[(2)] If $x\in\Omega_T(x)$, write $x\in\Omega\,(T,X)$\index{$\Omega\,(T,X)$}, then $x$ is called a \textit{nonwandering point} of $(T,X)$. If $\Omega\,(T,X)=X$, then $(T,X)$ itself is called a \textit{nonwandering} semiflow.
\end{enumerate}

Clearly,
\begin{itemize}
\item Both $\Omega_T(x)$ and $\Omega\,(T,X)$ are closed subsets of $X$; moreover, they are invariant whenever $T$ is a group or an abelian semigroup.
\end{itemize}

If $x$ is a nonwandering point of $(T,X)$, then for all neighborhood $U$ of $x$, there is some $t\in T$ outside any compact subset of $T$ such that $U\cap tU\not=\emptyset$.
\end{sn}

\begin{lem}\label{lem2.4}
Let $(T,X)$ be a semiflow with $T$ a locally compact non-compact semigroup and with $X$ a locally compact connected Hausdorff space. Then $(T,X)$ is TT if and only if the prolongation of every point of $X$ is $X$, i.e., $\Omega_T(x)=X\ \forall x\in X$.
\end{lem}

\begin{proof}
Let $(T,X)$ be TT and $x\in X$. Let $y\in X$ and let $U_x$ and $V_y$ be neighborhoods of $x$ and $y$ in $X$, respectively. Let $K$ be a compact neighborhood of $e$ in $T$.

Since $X$ is locally compact, we can assume $U_x$ is compact. Since $X$ is connected, we may suppose $V_y$ is open but not closed. Then as $KU_x$ is compact, it follows that $V_y\setminus KU_x\not=\emptyset$ so by TT we can conclude that $N_T(U,V)\cap(T\setminus K)\not=\emptyset$. Thus, there are $t\in T\setminus K$ and $x^\prime\in U_x$ with $tx^\prime\in V_y$. So $y\in\Omega_T(x)$.

Conversely, assume $\Omega_T(x)=X$ for all $x\in X$ and let $U,V$ be two non-empty open subsets of $X$. Take some $x\in U$ and then by $\Omega_T(x)\cap V\not=\emptyset$ it follows that $N_T(U,V)\not=\emptyset$.
\end{proof}

Thus, under the same situation of Lemma~\ref{lem2.4}, if $(T,X)$ is TT, then every point of $X$ is nonwandering. In particular, we can obtain the following.

\begin{cor}\label{cor2.5}
Let $(T,X)$ be a semiflow with $T$ a locally compact non-compact semigroup and with $X$ a topological manifold. If $(T,X)$ is TT, then $(T,X)$ is nonwandering.
\end{cor}

In view of the question we are concerned with, the following corollary of Lemma~\ref{lem2.4} is somewhat of interest.

\begin{cor}\label{cor2.6}
Let $(T,X)$ be PT with $T$ a locally compact non-compact semigroup and with $X$ a locally compact connected Hausdorff space. Then $(T,X)$ is TT iff $\Omega_T(x)\cap\mathrm{Tran}\,(T,X)\not=\emptyset$ for all $x\in X$.
\end{cor}

\begin{proof}
The necessity follows at once from Lemma~\ref{lem2.4}. To prove the sufficiency, let $D(x)$ be the invariant closed subset of $(T,X)$ defined as follows:
$y\in D(x)$ iff there are nets $\{x_n\}$ in $X$ and $\{t_n\}$ in $T$ with $x_n\to x$ and $t_nx_n\to y$.
Then $\Omega_T(x)\subseteq D(x)$ and so $D(x)\cap\mathrm{Tran}\,(T,X)\not=\emptyset$ for all $x\in X$. Thus $D(x)=X$ for all $x\in X$. Let $U,V$ be two non-empty open subsets of $X$. Take some $x\in U$ and then by $D(x)\cap V\not=\emptyset$ it follows that $N_T(U,V)\not=\emptyset$.
\end{proof}

\begin{remark}
The prolongation $\Omega_T(x)$ relies on the topology of $T$. For example, let $(t,x)\mapsto tx$ be a classical $C^0$-flow on a manifold $X$ with phase group $\mathbb{R}$. If $\mathbb{R}$ is under the discrete topology, then every point of $X$ is nonwandering so that this flow is nonwandering. Of course, there are nonwandering classical $C^0$-flows if $\mathbb{R}$ is under the usual topology.
\end{remark}
\section{Pre-recurrent transitive points of semiflows}\label{sec3}
Let $(T,X)$ be a semiflow on a topological space $X$ with phase semigroup $T$. We first introduce a notion which only works for semiflows with $T$ not groups. By $\textrm{cls}_XA$ we will denote the closure of a set $A$ in $X$.

\begin{sn}\label{def3.1}
A point $x\in X$ is called \textit{pre-recurrent} for $(T,X)$ if $x\in\textrm{cls}_XTsx$ for every $s\in T$. We could also define \textit{pointwise pre-recurrent}.
\end{sn}

\begin{sn}\label{def3.2}
An $x\in X$ is called a \textit{minimal point} of $(T,X)$ if $\textrm{cls}_XTx$ is a minimal subset of $(T,X)$; that is, $\textrm{cls}_XTy=\textrm{cls}_XTx$ for all $y\in\textrm{cls}_XTx$.
\end{sn}

The point $x=0$ in Example~\ref{exa1.3} is transitive but not pre-recurrent. By definitions the following lemma is evident and so we will omit its proof here.

\begin{lem}\label{lem3.3}
Each minimal point is pre-recurrent for $(T,X)$.
\end{lem}

\begin{sn}\label{def3.4}
\begin{enumerate}
\item A subset $A$ of $T$ is called \textit{thick} if for all compact subset $K$ of $T$, one can find an element $t\in T$ such that $Kt\subseteq A$.
\item A subset $S$ of $T$ is called \textit{syndetic} if there is a compact subset $K$ of $T$ such that $Kt\cap S\not=\emptyset$ for all $t\in T$.
\begin{itemize}
\item A subset of $T$ is syndetic iff it intersects non-voidly every thick set of $T$ (cf., e.g.,~\cite[Lemma~2.5]{AD}).
\end{itemize}
\item A point $x$ is called \textit{almost periodic} for $(T,X)$ if for all neighborhood $U$ of $x$, the set
\begin{equation*}
N_T(x,U)=\{t\in T\,|\,tx\in U\}
\end{equation*}
is syndetic in $T$.
\end{enumerate}
\end{sn}

It should be noted that since here $X$ is not necessarily a regular space, an almost periodic point need not be a minimal point. Conversely, since $X$ is not necessarily compact, a minimal point need not be an almost periodic point.

\begin{lem}\label{lem3.5}
Every almost periodic point is pre-recurrent for $(T,X)$.
\end{lem}

\begin{proof}
Let $x$ be an almost periodic point of $(T,X)$ and $U$ an arbitrary neighborhood of $x$; let $s\in T$. Since $N_T(x,U)$ is syndetic, there is some $k\in T$ with $ks\in N_T(x,U)$ so $U\cap Tsx\not=\emptyset$. Since $U$ is arbitrary, this implies $x\in\textrm{cls}_XTsx$.
\end{proof}

Next, based on pre-recurrent point we can easily obtain the following two simple criteria for TT of semiflows.

\begin{prop}\label{prop3.6}
If there is an $x\in\mathrm{Tran}\,(T,X)$ such that $x$ is pre-recurrent, then $(T,X)$ is TT.
\end{prop}

\begin{proof}
Let $x\in\mathrm{Tran}\,(T,X)$ be a pre-recurrent; then for all $s\in T$, $x\in\textrm{cls}_XTsx$ implies that $X=\textrm{cls}_XTsx$ and so $sx\in\mathrm{Trans}\,(T,X)$. Thus $\mathrm{Tran}\,(T,X)$ is dense in $X$ and so $(T,X)$ is TT.
\end{proof}

\begin{prop}\label{prop3.7}
Suppose $(T,X)$ is PT. Then $(T,X)$ is TT if and only if whenever (even for some) $x\in\mathrm{Tran}\,(T,X)$ and for all non-empty open sets $U,V$ in $X$, there are $s,t$ in $T$ such that $sx\in U$ and $tsx\in V$.
\end{prop}

\begin{proof}
The sufficiency is obvious. Now for the necessity, let $x$ be a transitive point and $U,V$ non-empty open sets. By TT, there is some $t\in T$ and a non-empty open set $W\subseteq U$ with $tW\subseteq V$. Finally by PT, there is an $s\in T$ such that $sx\in W$ so $tsx\in V$.
\end{proof}

In view of Example~\ref{exa1.3}, the pre-recurrence in Proposition~\ref{prop3.6} is important. More general than Proposition~\ref{prop3.6}, we can obtain the following result.

\begin{thm}\label{thm3.8}
Let $(T,X)$ be PT such that there exists an $x\in\mathrm{Tran}\,(T,X)$ which is a limit of pre-recurrent points. Then $(T,X)$ is TT.
\end{thm}

\begin{proof}
Let $x\in\mathrm{Tran}\,(T,X)$ which is a limit of pre-recurrent points, and let $U, V$ be non-empty open subsets of $X$. Since $x$ is a transitive point, there exist a neighborhood $W$ of $x$ and $s,t\in T$ such that $sW\subseteq V$ and $tx\in U$. Hence there exists a pre-recurrent point $z\in W$ such that $tz\in U$. This implies that there is some $\tau\in T$ such that $\tau(tz)\in W$ so that $s\tau U$ intersects $V$ non-voidly. Thus $(T,X)$ is TT by Definition~\ref{def1.1}. This proves Theorem~\ref{thm3.8}.
\end{proof}

\begin{cor}\label{cor3.9}
Let $(T,X)$ be a PT semiflow. If minimal points or almost periodic points are dense in $X$, then $(T,X)$ is TT.
\end{cor}

\begin{proof}
By Lemma~\ref{lem3.3} and Lemma~\ref{lem3.5}, the pre-recurrent points are dense in $X$. Then Corollary~\ref{cor3.9} follows from Theorem~\ref{thm3.8}.
\end{proof}
\section{PT and universally transitive semiflows}\label{sec4}
It has already been a well-known fact that
\begin{quote}
{\it If $f\colon Y\rightarrow Y$ is a continuous surjective transformation of a compact metric space $X$, then PT $\Rightarrow$ TT for $(f,Y)$} (cf.~\cite[Theorem~2.2.2]{KS}).
\end{quote}
In fact, we can obtain the following generalization to this result, in which the surjective condition is important according to Example~\ref{exa1.3}.

\begin{thm}\label{thm4.1}
Let $(T,X)$ be PT and surjective with $T$ an abelian semigroup. Then $\mathrm{Tran}\,(T,X)$ is invariant and thus $(T,X)$ is TT.
\end{thm}

\begin{proof}
Given $x_0\in\mathrm{Trans}\,(T,X)$ and $s\in T$, since $\textrm{cls}_XTsx_0=\textrm{cls}_XsTx_0\supseteq s\textrm{cls}_XTx_0=X$, thus we have $sx_0\in\mathrm{Trans}\,(T,X)$. This proves Theorem~\ref{thm4.1}.
\end{proof}

\begin{sn}\label{def4.2}
Let $(T,X)$ be a semiflow on a topological space $X$ with phase semigroup $T$.
\begin{enumerate}
\item Let $\textrm{Aut}\,(T,X)$ be the automorphism group of $(T,X)$; i.e., $\textrm{Aut}\,(T,X)$ is the group of all self-homeomorphisms $a$ of $X$ such that $at=ta$ for all $t\in T$.

 \item If $\textrm{Aut}\,(T,X)x=X$ for some $x\in X$ (so for all $x\in X$), then $(T,X)$ is called \textit{universally transitive} (UT\index{$\textrm{UT}$}) (or algebraically transitive in \cite{G48}).
\end{enumerate}
\end{sn}

\begin{note*}
If $(T,X)$ is UT, then $\textrm{Aut}\,(T,X)$ is referred to as \textit{transitive} on $X$ (cf.~\cite[Theorem~2.13]{Aus}).
\end{note*}

Using UT condition instead of the one that each $t\in T$ is surjective, we can obtain the following corollary.

\begin{cor}\label{cor4.3}
If $(T,X)$ is PT and UT with $T$ an abelian semigroup, then $\mathrm{Tran}\,(T,X)$ is invariant and thus $(T,X)$ is TT.
\end{cor}

\begin{proof}
By Theorem~\ref{thm4.1}, it is sufficient to show $tX=X$ for all $t\in T$. In fact, since $atX=taX=tX$ for all $t\in T$ and $a\in \textrm{Aut}\,(T,X)$, hence $atx\in tX$ for all $a\in\textrm{Aut}\,(T,X)$. So by UT, $tX=X$.
\end{proof}

We notice here that TT $+$ PT $\not\Rightarrow$ UT even for flows with compact metric phase spaces. Let's see such a simple example as follows.

\begin{exa}\label{exa4.4}
Let $X=\mathbb{R}\cup\{\infty\}$ be the one-point compactification of $\mathbb{R}$ with the usual topology and let $T=(\mathbb{R},+)$; define $\pi\colon T\times X\rightarrow X$, $(t,x)\mapsto t+x$, which is of course a flow. Then $(T,X)$ is TT and PT with $\mathrm{Tran}\,(T,X)=\mathbb{R}$ but $\textrm{Aut}\,(T,X)\infty=\{\infty\}\not=X$. Thus $(T,X)$ is not UT.
\end{exa}

Under UT condition, if our phase space $X$ is compact Hausdorff or compact metric, then we can gain more. First, let's recall a classical theorem of Gottschalk.
\begin{quote}
\textbf{Gottschalk's theorem} (cf.~\cite[Theorem~7]{G48}). {\it Let $(T,X)$ be a PT flow on a compact metric space with $T$ an abelian group. Then
$(T,X)$ is UT iff $(T,X)$ is equicontinuous.}
\end{quote}
Next we will generalize Gottschalk's theorem for semiflows. For this, we need to introduce some concepts and lemmas for self-closeness of this note.

\begin{sn}\label{def4.5}
Let $(T,X)$ be a semiflow on a compact Hausdorff space $X$ with the compatible symmetric uniform structure $\mathscr{U}_X$.
\begin{enumerate}
\item We say $(T,X)$ is \textit{distal} if given $x,y\in X$ with $x\not=y$, there is an $\alpha\in\mathscr{U}_X$ such that $(tx,ty)\not\in\alpha$ for all $t\in T$. Thus if $(T,X)$ is distal, then for two different initial points $x,y\in X$, their orbits $Tx$ and $Ty$ are synchronously far away.

\item $(T,X)$ is called \textit{equicontinuous} in case given $\varepsilon\in\mathscr{U}_X$, there is some $\delta\in\mathscr{U}_X$ such that if $(x,y)\in\delta$ then $(tx,ty)\in\varepsilon$ for all $t\in T$.

\item We say $x\in X$ is an \textit{equicontinuous point} of $(T,X)$, denoted $x\in\mathrm{Equi}\,(T,X)$, if given $\varepsilon\in\mathscr{U}_X$, there is $\delta\in\mathscr{U}_X$ such that $(tx,ty)\in\varepsilon\ \forall t\in T$ whenever $(x,y)\in\delta$.
\end{enumerate}
\end{sn}

\begin{lem}[{cf.~\cite[Lemma~1.6]{AD}}]
If $\mathrm{Equi}\,(T,X)=X$ with $X$ a compact Hausdorff space, then $(T,X)$ is equicontinuous.
\end{lem}

\begin{lem}[{cf.~\cite[Lemma~3.3]{KM} and \cite[Lemma~4.1]{DX}}]\label{lem4.7A}
If $(T,X)$ is a TT semiflow with $X$ a compact Hausdorff space, then $\mathrm{Equi}\,(T,X)\subseteq\mathrm{Tran}\,(T,X)$.
\end{lem}

It is a well-known basic fact that
\begin{quote}
{\it An equicontinuous flow is minimal if and only if it is PT} (cf., e.g.~\cite[p.~37]{Aus}).
\end{quote}
But this is not the case in semiflow situation. Let's see a simple example.

\begin{exa}\label{exa4.8}
Let $X=\{a,b\}$ and $f\colon a\mapsto b\mapsto b$; then the cascade $(f,X)$, which induces a $\mathbb{Z}_+$-action, is equicontinuous and PT with $\mathrm{Tran}\,(f,X)=\{a\}$, but it is not minimal.
\end{exa}

However if we consider TT instead of PT, then by Lemma~\ref{lem4.7A} we can obtain the following.

\begin{lem}\label{lem4.9}
Let $(T,X)$ be an equicontinuous semiflow on a compact Hausdorff space; then $(T,X)$ is TT if and only if it is minimal.
\end{lem}

\begin{proof}
If $(T,X)$ is minimal, then it is obviously TT. Conversely, if it is TT, then by Lemma~\ref{lem4.7A}, $X=\mathrm{Equi}\,(T,X)\subseteq\mathrm{Tran}\,(T,X)$ so $\textrm{cls}_XTx=X$ for all $x\in X$ and thus $(T,X)$ is minimal.
\end{proof}

Recall that if $\mathrm{Equi}\,(T,X)$ is dense in $X$, then $(T,X)$ is called \textit{almost equicontinuous} (cf.~\cite{AAB, G03}).
It should be noticed that if we relax ``equicontinuous'' by ``almost equicontinuous'', then the above statement is false even for flows as will be shown by the following example.

\begin{exa}\label{exa4.10}
We now construct a non-minimal cascade $(f,X)$ with $\mathrm{Equi}\,(f,X)=\mathrm{Tran}\,(f,X)$ dense. Let $X$ be the compact metric space with
\begin{gather*}
X=\{0,1\}\cup\left\{2^{-2^n}\,|\,n=0,1,2,\dotsc\right\}\cup\left\{2^{-1/2^n}\,|\,n=1,2,\dotsc\right\}
\end{gather*}
and $f\colon x\mapsto x^2$.
Clearly, $(f,X)$ is TT and PT such that $\mathrm{Equi}\,(f,X)=\mathrm{Tran}\,(f,X)=X\setminus\{0,1\}$ is dense in $X$. But $(f,X)$ is not minimal as a flow.
\end{exa}

Of course, if we relax ``equicontinuous'' by ``almost equicontinuous'' and meanwhile we strengthen ``TT'' by ``ST'' (cf.~Definition~\ref{def6.1}), then the statement of Lemma~\ref{lem4.9} still holds by Theorem~\ref{thm6.12} in $\S\ref{sec6}$.

\begin{lem}[{cf.~\cite[Corollary~3.3]{AD}}]\label{lem4.11}
If $(T,X)$ is a minimal semiflow on a compact Hausdorff space $X$ with $T$ abelian, then $(T,X)$ is surjective.
\end{lem}

\begin{lem}[{cf.~\cite[Theorem~1.13]{AD}}]\label{lem4.12}
Let $(T,X)$ be a semiflow on a compact Hausdorff space $X$ with phase semigroup $T$. Then:
\begin{enumerate}
\item[$(1)$] If $(T,X)$ is distal, then it is invertible.
\item[$(2)$] If $(T,X)$ is equicontinuous surjective, then it is distal.
\item[$(3)$] If $(T,X)$ is invertible equicontinuous, then $(\langle T\rangle, X)$ is an equicontinuous flow.
\end{enumerate}
\end{lem}

\begin{lem}[{\cite{GH,Aus}}]\label{lem4.13}
Let $\langle\varphi_n\colon X\rightarrow Y\rangle_{n=1}^\infty$ be a sequence of continuous functions on a Baire space $X$ to a metric space $Y$, which converges pointwise to a function $\varphi\colon X\rightarrow Y$. Let $E$ be the set of all $x\in X$ such that $\varphi_n\to\varphi$ uniformly at $x$. Then $E$ is a residual subset of $X$.
\end{lem}

Theorem~\ref{thm4.14} below gives us a sufficient and necessary condition for equicontinuity of TT semiflow with abelian phase semigroup in terms of UT. However, it should be mentioned that (1) of Theorem~\ref{thm4.14} in the important special case that $(T,X)$ is a \textit{flow} is originally due to Fort 1949~\cite{For} (also cf.~\cite[Theorem~9.36]{GH} and \cite[Theorem~2.13]{Aus}).

\begin{thm}\label{thm4.14}
Let $(T,X)$ be a semiflow on a compact metric space $X$ with phase semigroup $T$. Then the following two statements hold.
\begin{enumerate}
\item[$(1)$] If $(T,X)$ is UT, then $(T,X)$ is equicontinuous invertible.
\item[$(2)$] If $(T,X)$ is TT equicontinuous with $T$ abelian, then $(T,X)$ is UT.
\end{enumerate}
\end{thm}

\begin{note}
In view of Example~\ref{exa4.8}, PT $+$ Equicontinuous $+$ Abelian phase semigroup $\not\Rightarrow$ UT, in general semiflow setting.
\end{note}

\begin{note}\label{n2}
In fact, it is easy to verify that (2) of Theorem~\ref{thm4.14} also holds on compact Hausdorff phase space $X$ by using Ellis semigroup.
\end{note}

\begin{proof}
(1) Let $(T,X)$ be UT. Without loss of generality, we can regard $T$ as a subset of $C(X,X)$ the continuous self-maps of $X$, provided with the topology of uniform convergence.

Let $\langle t_n\rangle$ be any sequence in $T$. Choose a point $x_0\in X$. Some subsequence $\langle t_ix_0\rangle$ of $\langle t_nx_0\rangle$ converges, for $X$ is a compact metric space. Then we may suppose $\lim_{i\to\infty}t_ix_0=y_0$. If $x\in X$ and if $a\in\mathrm{Aut}\,(T,X)$ such that $a x_0=x$, then
$a y_0=\lim_{i\to\infty}at_ix_0=\lim_{i\to\infty}t_iax_0=\lim_{i\to\infty}t_ix$.
Hence the sequence $\langle t_i\rangle$ converges pointwise to some function $\varphi\colon X\rightarrow X$. Since $X$ is a Baire space, then by Lemma~\ref{lem4.13} there exists some point $x_1\in X$ such that $t_i\to\varphi$ uniformly at $x_1$.

Since $t_ia=at_i\ \forall a\in\mathrm{Aut}\,(T,X)$, hence $\varphi a=a\varphi\ \forall a\in\mathrm{Aut}\,(T,X)$. Then by UT of $(T,X)$, we can see that $\varphi$ is continuous on $X$. In fact, we need to show that $t_i\to\varphi$ uniformly on $X$. For this, let $d$ be a compatible metric on $X$ and let $x\in X$ and $\varepsilon>0$. Let $a\in\mathrm{Aut}\,(T,X)$ with $ax_1=x$ and let $\delta>0$ such that if $d(w,z)<\delta$, then $d(aw,az)<\varepsilon$. Let $V_{x_1}$ be a neighborhood of $x_1$ and $i_0$ a positive integer such that if $i\ge i_0$ and $w\in V_{x_1}$, then $(t_iw,\varphi w)<\delta$. Hence $d(t_iaw,\varphi aw)=d(at_iw,a\varphi w)<\varepsilon$. Thus if $y\in U=a[V_{x_1}]$, then $d(t_iy,\varphi y)<\varepsilon$ for all $i\ge i_0$. Therefore, $t_i\to\varphi$ uniformly on $X$.

This shows that every sequence $\langle t_n\rangle$ in $T$ has a uniformly convergent subsequence. Whence $T$ is relatively compact in $C(X,X)$. Therefore $T$ is equicontinuous on $X$ by the Ascoli-Arzel\`{a} theorem.
Finally by UT, each $t\in T$ is a surjection of $X$ as in the proof of Corollary~\ref{cor4.3}. Thus $(T,X)$ is invertible by Lemma~\ref{lem4.12}. So $(T,X)$ is equicontinuous surjective.

(2) Let $(T,X)$ be equicontinuous TT with $T$ an abelian semigroup. Lemma~\ref{lem4.9} follows that $(T,X)$ is minimal and further by Lemma~\ref{lem4.11}, $(T,X)$ is equicontinuous surjective. Then by Lemma~\ref{lem4.12}, $(T,X)$ is equicontinuous invertible. Further by Lemma~\ref{lem4.12}, it follows that $(\langle T\rangle,X)$ is an equicontinnous PT flow with $\langle T\rangle$ an abelian group. Gottschalk's theorem follows that $(\langle T\rangle,X)$ is UT. Since $\mathrm{Aut}\,(\langle T\rangle,X)\subseteq\mathrm{Aut}\,(T,X)$, thus $(T,X)$ is UT.

This thus concludes Theorem~\ref{thm4.14}.
\end{proof}

Since PT $\Rightarrow$ TT in flows, hence we now can generalize Gottschalk's theorem from flows to semiflows as follows:

\begin{cor}\label{cor4.15}
Let $(T,X)$ be a TT semiflow on a compact metric space $X$ with $T$ an abelian semigroup. Then $(T,X)$ is UT iff it is equicontinuous.
\end{cor}

We note that the metric on $X$ has played an important role in the proof of (1) of Theorem~\ref{thm4.14}. However, if $X$ is only a compact Hausdorff space non-metrizable, then what can we say?

Let $(T,X)$ be a minimal UT semiflow with compact Hausdorff phase space $X$. Then given $x,y\in X$, there is an $a\in\textrm{Aut}\,(T,X)$ such that $y=a(x)$. This implies that $(x,y)$ is almost periodic. Indeed, if $t_nx\to x^\prime$ (i.e. $t_n(x,y)\to(x^\prime,a(x^\prime))$) then there is $s_nx^\prime\to x$ (so $s_n(x^\prime, a(x^\prime))\to (x,y)$) for some net $\{s_n\}$ in $T$ since $(T,X)$ is minimal. Thus $(T,X)$ is distal.

For non-minimal case, we can obtain the following, whose proof may be simplified by using Ellis' semigroup \cite[pp.~15--22]{E69}.

\begin{thm}\label{thm4.16}
Let $(T,X)$ be a $($resp.~PT$)$ semiflow on a compact Hausdorff space $X$ with phase semigroup $T$. If $(T,X)$ is UT, then $(T,X)$ is $(\textrm{resp.~minimal})$ distal.
\end{thm}

\begin{proof}
For simplicity, write $H=\mathrm{Aut}\,(T,X)$. Then $Hx=X$ for all $x\in X$ by UT. To be contrary, assume $(T,X)$ is not distal; then there are two points $y,w\in X$ with $y\not=w$ such that there is a net $\{t_n\}$ in $T$ with $\lim_nt_ny=\lim_nt_nw=z$, for $X$ is compact.

Let $X^X$ be the compact Hausdorff space of all functions, continuous or not, from $X$ to itself with the pointwise convergence topology. Let $E$ be the closure of $T$ in $X^X$, where we identify each $t\in T$ with the transition map $x\mapsto tx$ of $X$ to $X$ associated with $(T,X)$.
Then $E_{y,w}$, defined by $E_{y,w}=\left\{p\in E\,|\,p(y)=p(w)\right\}$,
is a non-empty semigroup with the composition of maps such that $E_{y,w}$ is compact Hausdorff and for all $q\in E$, $R_q\colon p\mapsto pq$ is continuous under the pointwise topology. Whence there is an element $u\in E_{y,w}$ with $u^2=u$ (cf.~\cite[Lemma~2.9]{E69}). Clearly, $hu=uh\ \forall h\in H$. Now let $x\in X$. Then $Hx=Hux$. Hence there exists $h\in H$ with $hx=ux$. Then $hux=uhx=u^2x=ux=hx$ implies $ux=x$. Thus $u\in E_{y,w}$ is the identity so that $y=w$ a contradiction.

This shows that $(T,X)$ is distal if it is UT. Because a distal point is almost periodic, $(T,X)$ is minimal if it is PT and UT. Thus proves Theorem~\ref{thm4.16}.
\end{proof}

Then by Theorem~\ref{thm4.16} combining with Furstenberg~\cite{F63}, we can easily obtain the following consequence.

\begin{cor}\label{cor4.17}
If $(T,X)$ is a UT semiflow on a compact Hausdorff space $X$, then it admits an invariant Borel probability measure.
\end{cor}

Therefore by Theorem~\ref{thm4.16}, it follows that Ellis' `two circle' minimal set \cite[Example~5.29]{E69} is not UT; for otherwise, it would be distal.

In addition, since a minimal semiflow is TT, hence PT $+$ UT $\Rightarrow$ TT on compact Hausdorff phase spaces by Theorem~\ref{thm4.16}.

\section{Almost right C-semigroup actions}\label{sec5}
We will first introduce a kind of phase semigroup, which includes the two important special cases: $T=(\mathbb{R}_+,+)$ and $T=(\mathbb{Z}_+,+)$ equipped respectively with the usual topologies.

Recall that a topological semigroup $T$ is called a \textit{right C-semigroup}~\cite{KM} if $T\setminus Ts$ is relatively compact in $T$ for all $s\in T$.
In particular, $T\setminus Ts$, for $s\in T$, is a finite set if $T$ is a discrete right C-semigroup semigroup like $T=\mathbb{Z}_+$.

\begin{sn}[{cf.~\cite{Dai}}]\label{def5.1}
A topological semigroup $T$ is called an \textit{almost right C-semigroup} if and only if
$\{t\in T\,|\,\textrm{cls}_T(T\setminus Tt)\textrm{ is compact in }T\}$
is dense in $T$.
\end{sn}

Clearly each topological group is an almost right C-semigroup. See \cite[Examples~0.5]{Dai} for some examples of almost right C-semigroups which are not right C-semigroups.

\begin{thm}\label{thm5.2}
Assume $(T,X)$ is a PT semiflow satisfying the following two conditions:
\begin{enumerate}
\item[$(a)$] $\mathrm{Int}_XTx_0=\emptyset$ for some $x_0\in\mathrm{Tran}\,(T,X)$ and
\item[$(b)$] $T$ is an almost right \textit{C}-semigroup.
\end{enumerate}
Then $\mathrm{Tran}\,(T,X)$ is dense in $X$ and hence $(T,X)$ is TT.
\end{thm}

\begin{proof}
Let $x_0\in\mathrm{Tran}\,(T,X)$ and let $U$ be any non-empty open subset of $X$. We simply write $G$ for the dense set $\{t\in T\,|\,\textrm{cls}_T(T\setminus Tt)\textrm{ is compact in }T\}$. Then $\textrm{cls}_X{Gx_0}=\textrm{cls}_X{Tx_0}=X$.

Given any $s\in G$ with $s\not=e$, we will show that $U\cap\textrm{cls}_XTsx_0\not=\emptyset$. To be contrary, assume that $U\cap\textrm{cls}_XTsx_0=\emptyset$. Set $K=\textrm{cls}_T(T\setminus Ts)$, which is compact in $T$ by condition $(b)$. Since $T=K\cup Ts$ and so
\begin{equation*}
X=\textrm{cls}_XTx_0=\textrm{cls}_X\left(Kx_0\cup Tsx_0\right)=Kx_0\cup\textrm{cls}_XTsx_0,
\end{equation*}
hence $U\subseteq Kx_0$, which contradicts condition $(a)$. Thus $\textrm{cls}_XTsx_0=X$ for all $s\in G$ and then $Gx_0\subseteq\mathrm{Tran}\,(T,X)$.
This proves Theorem~\ref{thm5.2}.
\end{proof}

We notice here that condition $(a)$ is very important for the consequences of Theorem~\ref{thm5.2} and Corollary~\ref{cor5.3} below, which implies that $(T,X)$ has no isolated orbit if $T$ is a group. Otherwise, Example~\ref{exa4.8} is a counterexample.

\begin{cor}\label{cor5.3}
Let $(T,X)$ be a PT equicontinuous semiflow with $X$ a compact Hausdorff space such that:
\begin{enumerate}
\item[$(a)$] $\mathrm{Int}_XTx_0=\emptyset$ for some $x_0\in\mathrm{Tran}\,(T,X)$ and
\item[$(b)$] $T$ is an abelian almost right \textit{C}-semigroup.
\end{enumerate}
Then $(T,X)$ is UT, minimal, and distal.
\end{cor}

\begin{proof}
By Theorem~\ref{thm5.2}, $(T,X)$ is TT equicontinuous with $T$ abelian. So $(T,X)$ is UT by Note~\ref{n2} to Theorem~\ref{thm4.14}. Then by Theorem~\ref{thm4.16}, $(T,X)$ is minimal distal.
\end{proof}

Clearly a discrete semigroup is a right C-semigroup if and only if it is an almost right C-semigroup. The second part of the following Proposition~\ref{prop5.4} is just \cite[(1) of Proposition~3.2]{KM}.

\begin{prop}\label{prop5.4}
Let $(T,X)$ be a PT semiflow, where $X$ has no isolated point. If $T$ is a right \textit{C}-semigroup under the discrete topology, then $\mathrm{Tran}\,(T,X)$ is dense in $X$ and hence $(T,X)$ is TT.
\end{prop}

\begin{proof}
In the above proof of Theorem~\ref{thm5.2}, $K=T\setminus Ts$ is finite and $U\subseteq Kx_0$ implies that $X$ contains isolated points. This contradiction completes the proof.
\end{proof}

\begin{cor}\label{cor5.5}
Let $(T,X)$ be a PT semiflow with $T$ a countable semigroup, where $X$ has no isolated point. If $T$ is an almost right \textit{C}-semigroup, then $\mathrm{Tran}\,(T,X)$ is dense in $X$ and hence $(T,X)$ is TT.
\end{cor}

\begin{proof}
Since $T$ is a countable semigroup, then $\mathrm{Int}_XTx=\emptyset$, for all $x\in X$ for $X$ has no isolated points. Then Corollary~\ref{cor5.5} follows from Theorem~\ref{thm5.2}.
\end{proof}
\section{Syndetic transitivity and syndetic sensitivity}\label{sec6}
We will consider another kind of transitivity, which is important for chaos of semiflows (see, e.g., \cite{DT} and Theorem~\ref{thm6.12}).

\begin{sn}\label{def6.1}
$(T,X)$ is called \textit{syndetically transitive} (ST\index{$\textrm{ST}$}) if $N_T(U,V)$ is syndetic in $T$ for all non-empty open sets $U,V$ in $X$.
\end{sn}

Note that ST is TT, but the converse is false. For example, $\mathbb{R}\times X\rightarrow X,\ (t,x)\mapsto t+x$, where $X=\mathbb{R}\cup\{\infty\}$ is the one-point compactification as in Example~\ref{exa4.4}, is TT but not ST.

On a Polish space $X$, $(T,X)$ is TT if and only if $\mathrm{Tran}\,(T,X)$ is a residual subset of $X$ (cf.~\cite[(2) of Proposition~3.2]{KM} and \cite{DT}). However, this is not the case for semiflows with non-separable phase spaces (cf.~\cite[Example~4.17]{EE}).

In fact, in general, $\mathrm{ST}\not\Rightarrow\mathrm{PT}$, for flows on compact \textit{non-separable} Hausdorff spaces. Let's construct such an example as follows:

\begin{exa}\label{exa6.2}
Let $X=Y^T$ be the space of all functions $f\colon T\rightarrow Y$, continuous or not, equipped with the pointwise convergence topology, where $Y$ is a compact Hausdorff space and $T$ is an infinite discrete group. Then $X$ is a compact Hausdorff space.
Given $t\in T$ and $f\in X$, define
$f^t\colon T\rightarrow Y$ by $\tau\mapsto f(\tau t)$.
We now define the flow on $X$ with the phase group $T$ as follows:
$T\times X\rightarrow X$ by $(t,f)\mapsto f^t$.

(1) First we can assert that $(T,X)$ is ST. Indeed, for all non-empty open subsets $U,V$ of $Y$ and $\tau_1,\dotsc,\tau_n, s_1,\dotsc,s_n$ in $T$, let
\begin{gather*}
\mathcal{U}=\left[\{\tau_1,\dotsc,\tau_n\},U\right]=\left\{f\in X\,|\,f(\tau_i)\in U, i=1,\dotsc,n\right\}\\ \mathcal{V}=\left[\{s_1,\dotsc,s_n\},V\right].
\end{gather*}
Let
\begin{equation*}
T_0=\left\{s_i^{-1}\tau_j\,|\,i=1,\dotsc,n;\ j=1,\dotsc,n\}\cup\{s_i^{-1}s_j\,|\,i=1,\dotsc,n;\  j=1,\dotsc,n\right\}
\end{equation*}
and
\begin{equation*}
T_1=T\setminus T_0.
\end{equation*}
Since $T$ is an infinite discrete group and $T_0$ is finite, it is easy to check that $T_1$ is syndetic in $T$. (In fact, if $K=\{e\}\cup t_0T_0^{-1}$ for some $t_0\in T_1$, then $Kt\cap T_1\not=\emptyset$ for all $t\in T=T_0\cup T_1$ so $T_1$ is syndetic in $T$.)
Thus for all $\bt\in T_1$,
\begin{gather*}
\{s_1\bt,\dotsc,s_n\bt\}\cap\{\tau_1,\dotsc,\tau_n, s_1,\dotsc,s_n\}=\emptyset.
\end{gather*}
Now choose $f\in X$ such that $f(\tau_i)\in U$ and $f(s_i\bt)\in V$ for $1\le i\le n$. Thus $f\in\mathcal{U}$ and $f^{\bt}\in\mathcal{V}$ so that $N_T(\mathcal{U},\mathcal{V})\supseteq T_1$. Thus
\begin{itemize}
\item {\it $(T,X)$ is ST.}
\end{itemize}

(2) Now we choose $Y$ a non-separable space (so $Y$ has no countable dense subset) and let $T$ be a countable infinite discrete group. Then $X$ has no countable dense subset.
Because $X$ is not separable and $T$ is countable, it follows that:
\begin{itemize}
\item {\it $(T,X)$ is not PT; i.e., $\mathrm{Tran}\,(T,X)=\emptyset$.}
\end{itemize}
This completes the construction of our Example~\ref{exa6.2}.
\end{exa}

In view of Example~\ref{exa6.2} we now ask two questions:

\begin{quote}
\begin{enumerate}
\item {\it If $(T,X)$ is a TT and pointwise almost periodic flow/semiflow on a compact non-separable Hausdorff space, is it PT $($or equivalently minimal$)$?}
\item {\it If $(T,X)$ is minimal and ST with $X$ a locally compact, non-compact, Hausdorff space, is $(T,X)$ pointwise almost periodic?}
\end{enumerate}
\end{quote}

Let $T$ be an infinite discrete group. Recall that $(T,X)$ is called \textit{strongly mixing} if $N_T(U,V)$ is co-finite for all non-empty open sets $U,V$ in $X$. By the arguments in Example~\ref{exa6.2}, $(T,X)$ is strongly mixing based on any compact Hausdorff space $Y$.

\begin{lem}[{\cite[Lemma~2.3]{DT}}]\label{lem6.3}
If $(T,X)$ is TT with dense almost periodic points, then it is ST.
\end{lem}

Thus by Corollary~\ref{cor3.9} together with Lemma~\ref{lem6.3}, we can easily obtain the following.

\begin{cor}
If $(T,X)$ is PT with dense almost periodic points, then $(T,X)$ is ST.
\end{cor}

\begin{snAus}
In the remainder of this section, let $X$ be a uniform Hausdorff space with a symmetric uniform structure $\mathscr{U}_X$. For $x\in X, A\subset X$ and $\varepsilon\in\mathscr{U}_X$, we write
\begin{equation*}
\varepsilon[x]=\{y\in X\,|\,(x,y)\in\varepsilon\}\quad \textrm{and}\quad \varepsilon[A]={\bigcup}_{x\in A}\varepsilon[x].
\end{equation*}
Given $(T,X)$ and $\varepsilon,\delta\in\mathscr{U}_X$, the ``$(\varepsilon,\delta)$-stable-time set'' at a point $x\in X$ is defined as follows:
\begin{equation*}
T_{\varepsilon\textrm{-}\delta}^s(x)=\{t\in T\,|\,t(\delta[x])\subseteq\varepsilon[tx]\}.
\end{equation*}
\end{snAus}

Next we will consider a simple application of ST in chaos. For this, we first need to introduce some notions.

\begin{sn}[{cf.~\cite{KM,MM,DT,WZ}}]\label{def6.6}
Let $(T,X)$ be a semiflow on $(X,\mathscr{U}_X)$ with phase semigroup $T$. Then:
\begin{enumerate}
\item $(T,X)$ is called \textit{sensitive} if there exists an $\varepsilon\in\mathscr{U}_X$ such that for all $x\in X$ and all $\delta\in\mathscr{U}_X$, there is an $x^\prime$ with $(x,x^\prime)\in\delta$ and $t(x,x^\prime)\not\in\varepsilon$ for some $t\in T$; that is, $T_{\varepsilon\textrm{-}\delta}^s(x)\not=T$.

\item $(T,X)$ is called \textit{syndetically sensitive} if there exists an $\varepsilon\in\mathscr{U}_X$ such that for all $x\in X$ and $\delta\in\mathscr{U}_X$,
$T_{\varepsilon\textrm{-}\delta}^s(x)$ is not thick in $T$.

\item $(T,X)$ is said to be \textit{pointwise thickly stable} if given $\varepsilon\in\mathscr{U}_X$ and $x\in X$, one can find a $\delta\in\mathscr{U}_X$ such that $T_{\varepsilon\textrm{-}\delta}^s(x)$ is thick in $T$.

\item $(T,X)$ is called \textit{pointwise equicontinuous} if given $\varepsilon\in\mathscr{U}_X$ and $x\in X$, one can find a $\delta\in\mathscr{U}_X$ such that $T_{\varepsilon\textrm{-}\delta}^s(x)=T$; see 3 of Definition~\ref{def4.5}.

\item $(T,X)$ is called \textit{uniformly almost periodic} if given $\varepsilon\in\mathscr{U}_x$, there is a syndetic subset $A$ of $T$ such that $Ax\subset\varepsilon[x]$ for all $x\in X$.
\end{enumerate}
\end{sn}

Clearly, syndetically sensitive $\Rightarrow$ sensitive, and equicontinuous $\Rightarrow$ thickly stable. Moreover, by the classical ``Lebesgue covering lemma'', we can easily obtain the following uniformity:

\begin{lem}\label{lem6.7}
If $(T,X)$ is pointwise thickly stable with $X$ a compact Hausdorff space, then for every $\varepsilon\in\mathscr{U}_X$ there exists a $\delta\in\mathscr{U}_X$ such that $T_{\varepsilon\textrm{-}\delta}^s(x)$ is thick in $T$ for all $x\in X$.
\end{lem}

In addition, we will need another known result.

\begin{lem}[\cite{DX}]\label{lem6.8}
If $(T,X)$ is a semiflow on a compact Hausdorff space, then it is uniformly almost periodic iff it is equicontinuous surjective.
\end{lem}

By definitions, pointwise equicontinuous $($resp. thickly stable$)$ is much more stronger than `not to be sensitive' (resp. `not to be syndetically sensitive'). But we shall show that they are equivalent under `ST' condition.

The following lemma is essentially contained in Furstenberg's argument of \cite[p.~74\,--75]{Fur} for measure-preserving cascades. For the self-closeness, we will prove it here.

\begin{lem}\label{lem6.9}
Let $T\times(X,\mathscr{B},\mu)\rightarrow(X,\mathscr{B},\mu), (t,x)\mapsto tx$ be a measure-preserving semiflow of a probability space $(X,\mathscr{B},\mu)$, where $T$ is a discrete countable semigroup with $e\in T$. If $A\in\mathscr{B}$ with $\mu(A)>0$, then $\{t\in T\,|\,\mu(t^{-1}A\cap A)>0\}$ is syndetic in $T$.
\end{lem}

\begin{proof}
First of all, since $\mu(X)=1$ and $\mu(A)>0$, we note that there exists a finite subset $K$of $T$ such that
\begin{gather*}
\mu\left({\bigcup}_{t\in K} t^{-1}A\right)>\mu\left({\bigcup}_{t\in T}t^{-1}A\right)-\mu(A).
\end{gather*}
Then for all $s\in T$,
\begin{gather*}
\mu\left({\bigcup}_{t\in Ks} t^{-1}A\right)=\mu\left({\bigcup}_{t\in K} t^{-1}A\right)>\mu\left({\bigcup}_{t\in T} t^{-1}A\right)-\mu(A).
\end{gather*}
This implies that $\mu\left({\bigcup}_{t\in Ks} t^{-1}A\cap A\right)>0$ for all $s\in T$; for otherwise, $e\not\in Ks$ so that
$$
\mu\left({\bigcup}_{t\in T}t^{-1}A\right)\ge \mu(A)+\mu\left({\bigcup}_{t\in Ks} t^{-1}A\right)>\mu\left({\bigcup}_{t\in T} t^{-1}A\right),
$$
a contradiction. Thus, for all $s\in T$, there is some $t\in Ks$ with $\mu(t^{-1}A\cap A)>0$. This implies that $\{t\in T\,|\,\mu(t^{-1}A\cap A)>0\}$ is syndetic in $T$ by 2 of Definition~\ref{def3.4}.

The proof of Lemma~\ref{lem6.9} is thus completed.
\end{proof}

Let $f\colon X\rightarrow X$ be a continuous transformation on a compact metric space $X$. Recall that $(f,X)$ is called an \textit{E-system} if it is TT and admits an invariant Borel probability measure with full support. Clearly an \textit{E}-semiflow is surjective. Moreover, it is known that a non-minimal \textit{E}-system $(f,X)$ is sensitive \cite[Theorem~1.3]{GW} and $N_f(U,U)$ is syndetic in $\mathbb{Z}_+$ for all non-empty open set $U$ in $X$~\cite{GW2}.

\begin{sn}\label{def6.10}
$(T,X)$ is called an \textit{E-semiflow} if it is TT and for all non-empty open subset $U$ of $X$ there is an invariant Borel probability measure $\mu$ of $(T,X)$ with $\mu(U)>0$.
\end{sn}

Clearly, an \textit{E}-system of Glasner and Weiss can induce an \textit{E}-semiflow with phase semigroup $\mathbb{Z}_+$. However, since $X$ need not be second countable, so an \textit{E}-system in the sense of Definition~\ref{def6.10} is not necessarily an \textit{E}-system in the sense of Glasner and Weiss.

\begin{prop}\label{prop6.11}
Let $(T,X)$ be an \textit{E}-semiflow with $T$ a countable discrete semigroup. Then $(T,X)$ is ST.
\end{prop}

\begin{proof}
Let $U,V$ be non-empty open subsets of $X$. By TT, let $\tau\in T$ with $A:=U\cap \tau^{-1}V\not=\emptyset$. Then by Lemma~\ref{lem6.9}, it follows that $N_T(A,A)$ is syndetic in $T$. Since $\tau N_T(A,A)$ is syndetic in $T$ and $\tau N_T(A,A)\subseteq N_T(U,V)$, thus $(T,X)$ is ST.
\end{proof}

Therefore Theorem~\ref{thm6.12} below is a generalization and strengthening of the theorem of Glasner and Weiss on \textit{E}-systems.

The following result has already been proved by Miller and Money \cite[Theorem~4.4]{MM} with $X$ a compact metric space and by Wang and Zhong~\cite[Theorem~4.3]{WZ} with $X$ a compact Hausdorff space, using different approaches.

\begin{thm}\label{thm6.12}
If $(T,X)$ is a non-minimal ST semiflow with $X$ a uniform Hausdorff space, then it is syndetically sensitive.
\end{thm}

\begin{proof}
First of all, by non-minimality of $(T,X)$ there are $\eta\in\mathscr{U}_X, q\in X$ and a closed invariant subset $\Lambda$ of $(T,X)$ such that $\eta[q]\cap\eta[\Lambda]=\emptyset$. Assume $(T,X)$ is not syndetically sensitive; then for all $\varepsilon\ll\eta$, there are $x\in X$ and $\delta\in\mathscr{U}_X$ such that $T_{\varepsilon\textrm{-}\delta}^s(x)$ is a thick subset of $T$.

Since $(T,X)$ is ST, $N_T(\delta[x],\eta[q])$ is syndetic and so there exists a compact subset $K$ of $T$ such that $Kt\cap N_T(\delta[x],\eta[q])\not=\emptyset$. That is to say, for any $t\in T$, there is some $y\in\delta[x]$ such that $Kty\cap\eta[q]\not=\emptyset$.

Let $p\in\Lambda$ and since $\Lambda$ is invariant for $(T,X)$, there exists some $\gamma\in\mathscr{U}_X$ with $\gamma\ll\eta$ such that $K(\gamma[p])\subset\varepsilon[\Lambda]$.

Let $\mathcal{F}$ be the collection of non-empty compact subsets of $T$. Then since $T_{\varepsilon\textrm{-}\delta}^s(x)$ is a thick subset of $T$ and $K$ is compact in $T$, so for any $F\in\mathcal{F}$, there is some $t_F$ so that $(KF)t_F\subset T_{\varepsilon\textrm{-}\delta}^s(x)$.
Clearly, $T_0=\bigcup_{F\in\mathcal{F}}Ft_F$ is a thick subset of $T$.

Moreover, since $N_T(\delta[x],\gamma[p])$ is syndetic in $T$ by ST, $T_0\cap N_T(\delta[x],\gamma[p])\not=\emptyset$. Now having chosen $\tau\in T_0\cap N_T(\delta[x],\gamma[p])$, there exists some point $y_0\in\delta[x]$ such that $\tau y_0\in\gamma[p]$ so that $K\tau y_0\subset\varepsilon[\Lambda]$. By the choice of $\tau$ and the definition of $T_0$, it follows that $K\tau(\delta[x])\subset\eta[\Lambda]$. But, this contradicts that $K\tau\cap N_T(\delta[x],\eta[q])\not=\emptyset$.

Therefore, $(T,X)$ must be syndetically sensitive.
\end{proof}

Recall that if $T$ is a topological semigroup and if $A$ is a thick subset of $T$, then $As$ is thick in $T$ for all $s\in T$.
Now based on Theorem~\ref{thm6.12}, we can conclude the following dichotomy theorem for any ST semiflow on uniform spaces.

\begin{cor}\label{cor6.13}
Let $(T,X)$ be an ST semiflow with phase semigroup $T$. Then it is either syndetically sensitive or minimal thickly stable.
\end{cor}

\begin{proof}
Clearly if $(T,X)$ is syndetically sensitive, then it is never thickly stable. Next assume $(T,X)$ is not syndetically sensitive; and then by Theorem~\ref{thm6.12}, it follows that $(T,X)$ is minimal.

Now since $(T,X)$ is not syndetically sensitive, then for any $\varepsilon\in\mathscr{U}_X$, there are some $x\in X$ and $\delta\in\mathscr{U}_X$ such that $T_{\varepsilon\textrm{-}\delta}^s(x)$ is thick in $T$. Let $y\in X$ be arbitrary. Since $(T,X)$ is minimal, we can find some $\alpha\in\mathscr{U}_X$ and $s\in T$ such that $s(\alpha[y])\subseteq\delta[x]$. So $T_{\varepsilon\textrm{-}\alpha}^s(y)\supseteq T_{\varepsilon\textrm{-}\delta}^s(y)s$ is thick in $T$. Since $\varepsilon$ and $y$ both are arbitrary, thus $(T,X)$ is thickly stable.
\end{proof}

In fact, by Theorem~\ref{thm6.12}, we can obtain the following

\begin{cor}\label{cor6.14}
Let $(T,X)$ be ST with $T$ an almost right \textit{C}-semigroup. Then it is either sensitive or minimal equicontinuous.
\end{cor}

\begin{proof}
Assume $(T,X)$ is not sensitive. So it is not syndetically sensitive. By Theorem~\ref{thm6.12}, $(T,X)$ is minimal. Since $(T,X)$ is not sensitive, then for any $\varepsilon\in\mathscr{U}_X$, we can find some $x\in X$ and $\alpha\in\mathscr{U}_X$ such that $T_{\varepsilon\textrm{-}\alpha}^s(x)=T$. Given any $y\in X$, since $(T,X)$ is minimal, there are some $\delta^\prime\in\mathscr{U}_X$ and some $\tau\in \{t\in T\,|\,\textrm{cls}_T(T\setminus Tt)\textrm{ is compact in }T\}$ such that $\tau(\delta^\prime[y])\subseteq\alpha[x]$. In addition, since $T_0:=T\setminus T\tau$ is relatively compact by Definition~\ref{def5.1}, there is a $\delta\in\mathscr{U}_X$ with $\delta<\delta^\prime$ such that $t(\delta[y])\subseteq\varepsilon[ty]$ for all $t\in T_0$. Thus $T_{\varepsilon\textrm{-}\delta}^s(y)=T$. This shows that $(T,X)$ is equicontinuous.

Finally, if $(T,X)$ is sensitive, it is never minimal equicontinuous. Thus we have concluded Corollary~\ref{cor6.14}.
\end{proof}

If $(T,X)$ is TT with dense almost periodic points, then it is called an \textit{M-semiflow} (\cite{GW, DT}). By Lemma~\ref{lem6.3}, \textit{M}-semiflow is ST.
Thus Corollary~\ref{cor6.14} generalizes \cite[Theorem~1.41]{G03} that is for \textit{M-flow} on a compact metric space $X$ and \cite[Main result]{KM} that is for \textit{M-semiflow} on a Polish space with $T$ a right \textit{C}-semigroup by using different approaches.

The following corollary has already been observed in~\cite[Corollary~2.7]{DT}; but its proof presented in \cite{DT} is insufficient. Here we prove it using Theorem~\ref{thm6.12}.

\begin{cor}\label{cor6.15}
Let $(T,X)$ be ST with $X$ a compact Hausdorff space and with $T$ an abelian semigroup. Then $(T,X)$ is either sensitive or minimal uniformly almost periodic.
\end{cor}

\begin{proof}
If $(T,X)$ is sensitive, then it is not equicontinuous and so not uniformly almost periodic by Lemma~\ref{lem6.8}.

Now assume $(T,X)$ is not sensitive. Then by Theorem~\ref{thm6.12}, $(T,X)$ is minimal; otherwise it is syndetically sensitive. Moreover, for all $\varepsilon\in\mathscr{U}_X$, we can find some $x\in X$ and $\delta\in\mathscr{U}_X$ such that $T_{\varepsilon\textrm{-}\delta}^s(x)=T$. We will show that $(T,X)$ is uniformly almost periodic.

Let $\varepsilon\in\mathscr{U}_X$ be any given; and take an $\eta\in\mathscr{U}_X$ with ``$\eta\ll\varepsilon$'' and then we can choose $x_0\in X$ and $\delta\in\mathscr{U}_X$ such that $T_{\eta\textrm{-}\delta}^s(x_0)=T$. So $t(\delta[x_0])\subseteq\eta[tx_0]$ for all $t\in T$. Let $A=N_T(x_0,\delta[x_0])$, which is syndetic in $T$.
Now for any $a\in A$ and $s\in T$, we have $ax_0\in\delta[x_0]$ and then it follows that $(sx_0,asx_0)=(sx_0,sax_0)\in\eta$. Since $Tx_0$ is dense in $X$, hence $(z,tz)\in\varepsilon$ for all $z\in X$ and $t\in A$. Thus $(T,X)$ is uniformly almost periodic
This proves Corollary~\ref{cor6.15}.
\end{proof}

\subsection*{\textbf{Acknowledgments}}
This project was supported by National Natural Science Foundation of China (Grant Nos. 11431012 and 11271183) and PAPD of Jiangsu Higher Education Institutions.



\end{document}